\documentclass[a4paper,final]{amsart}

\usepackage{nicefrac}
\usepackage[foot]{amsaddr}

\newtheorem{theorem}{Theorem}

\newtheorem{lemma}{Lemma}

\newtheorem{remark}{Remark}
\newcommand{\intd}[1]{\left\lfloor#1\right\rfloor}
\newcommand{\intu}[1]{\left\lceil#1\right\rceil}

\newcommand{\dx}{\,\mathsf{d}x}

\title[]{Rounding the arithmetic mean value\\ of the square roots of the first $n$~integers}

\author{Thomas P. Wihler}
\address{Mathematics Institute, University of Bern, CH-3012 Bern, Switzerland}
\email{wihler@math.unibe.ch}

\begin{document}

\begin{abstract}
In this article we study the arithmetic mean value~$\Sigma(n)$ of the square roots of the first~$n$ integers. For this quantity, we develop an asymptotic expression, and derive a formula for its integer part which has been conjectured recently in the work of M.~Merca. Furthermore, we address the numerical evaluation of~$\Sigma(n)$ for large~$n\gg 1$.
\end{abstract}

\maketitle

The aim of this article is to derive an explicit formula for the integer part of the arithmetic mean value of the square roots of the first~$n$ integers. More precisely, we consider the sequence
\[
\Sigma(n)=\frac{1}{n}\sum_{k=1}^n\sqrt{k},\qquad n\in\mathbb{N},
\]
and show the following identity.

\begin{theorem}\label{thm:main}
For any~$n\in\mathbb{N}$, there holds that
\begin{equation}\label{eq:main}
\intd{\Sigma(n)}=\intd{A(n)},
\end{equation}
where we define the function
\begin{equation}\label{eq:B}
A(x)=\frac{2}{3}\sqrt{x+1}\left(1+\frac{1}{4x}\right),
\end{equation}
for~$x\ge 1$. Here, $\intd{\cdot}$ signifies the integer part of a positive real number.
\end{theorem}

This result is motivated by the recent work~\cite[see Conjecture~2]{Merca:17}, where Theorem~\ref{thm:main} has been conjectured.

\section{An asymptotic result}

In order to prove Theorem~\ref{thm:main}, we begin by deriving an asymptotic result for the sum of the square roots of the first~$n$ integers. Here, we employ an idea presented in~\cite{Merca:17}, which is based on using the trapezium rule for the numerical approximation of integrals. In this context, we also point to the related work~\cite{Shekatkar:13}, where upper and lower Riemann sums have been applied. In comparison to the analysis pursued in~\cite{Merca:17}, in the current paper, we use a different approach to control the error in the trapezium rule. Thereby, we arrive at a slightly sharper asymptotic representation for large~$n$. Incidentally, an asymptotic representation has been derived already in the early work~\cite{Ramanujan:00}.

\begin{theorem}\label{thm:r2}
For any~$\nu,n\in\mathbb{N}$, with~$\nu< n$, there holds
\[
\sum_{k=\nu}^{n}\sqrt{k}
=nA(n)-\frac23\sqrt{\nu}\left(\nu-\frac34\right)
-\frac{\delta_{\nu,n}}{24},
\]
with
\begin{equation}\label{eq:delta}
\sigma(\nu+2,n+2)
<\delta_{\nu,n}<\sigma(\nu,n),
\end{equation}
where
\[
\sigma(\nu,n)=
\begin{cases}
\nicefrac{3}{2}-n^{-\nicefrac12}&\text{if }\nu=1,\\
(\nu-1)^{-\nicefrac{1}{2}}-n^{-\nicefrac{1}{2}}&\text{if }\nu\ge 2.
\end{cases}
\] 
\end{theorem}

\begin{proof}
Let us consider the function~$f(x)=\sqrt{x}$, for~$x\ge1$. We interpolate it by a piecewise linear function~$\ell$ in the points~$x=1,2,\ldots$, i.e., for any~$k\in\mathbb{N}$ there holds~$\ell(k)=\sqrt{k}$, and~$\ell$ is a linear polynomial on the interval~$[k,k+1]$. We define the remainder term
\[
\widehat{\delta}_{\nu,n}:=\int_{\nu}^{n+1} \left(\sqrt{x}-\ell(x)\right)\dx
=\frac23\left((n+1)^{\nicefrac32}-\nu^{\nicefrac32}\right)
-\int_{\nu}^{n+1}\ell(x)\dx.
\]
Then, we note that
\[
\int_{\nu}^{n+1}\ell(x)\dx
=\frac12\sqrt{\nu}+\sum_{k=\nu+1}^{n}\sqrt{k}+\frac12\sqrt{n+1};
\]
this is the trapezium rule for the numerical integration of~$f$. Therefore,
\begin{align*}
\sum_{k=\nu+1}^{n}\sqrt{k}
&=\frac23\left((n+1)^{\nicefrac32}-\nu^{\nicefrac32}\right)
-\frac12\sqrt{\nu}
-\frac12\sqrt{n+1}
-\widehat{\delta}_{\nu,n}\\
&=\frac23\sqrt{n+1}\left(n+\frac14\right)
-\frac23\sqrt\nu\left(\nu+\frac34\right)
-\widehat{\delta}_{\nu,n}.
\end{align*}
Hence,
\begin{align*}
\sum_{k=\nu}^{n}\sqrt{k}
&=\frac23\sqrt{n+1}\left(n+\frac14\right)
-\frac23\sqrt\nu\left(\nu-\frac34\right)
-\widehat{\delta}_{\nu,n}\\
&=nA(n)-\frac23\sqrt\nu\left(\nu-\frac34\right)
-\widehat{\delta}_{\nu,n}.
\end{align*}
It remains to study the error term~$\widehat{\delta}_{\nu,n}$. For this purpose, applying twofold integration by parts, for~$k\in\mathbb{N}$, we note that
\begin{align*}
\int_{k}^{k+1} \left(\sqrt{x}-\ell(x)\right)\dx
&=-\frac{1}{8}\int_k^{k+1}\left(\sqrt{x}-\ell(x)\right)\frac{\mathsf{d}^2}{\mathsf{d}x^2}\left(1-4(x-k-\nicefrac12)^2\right)\dx\\
&=\frac{1}{32}\int_k^{k+1}x^{-\nicefrac{3}{2}}\left(1-4(x-k-\nicefrac12)^2\right)\dx\\
&<\frac{1}{32}k^{-\nicefrac{3}{2}}\int_k^{k+1}\left(1-4(x-k-\nicefrac12)^2\right)\dx\\
&=\frac{1}{48}k^{-\nicefrac{3}{2}}.
\end{align*}
Thus, if~$\nu\ge2$, we obtain
\begin{align*}
\widehat{\delta}_{\nu,n}
&=\sum_{k=\nu}^n\int_{k}^{k+1} \left(\sqrt{x}-\ell(x)\right)\dx\\
&<\frac{1}{48}\sum_{k=\nu}^nk^{-\nicefrac{3}{2}}\\
&<\frac{1}{48}\int_{\nu-1}^nx^{-\nicefrac{3}{2}}\dx\\
&=\frac{1}{24}\left((\nu-1)^{-\nicefrac{1}{2}}-n^{-\nicefrac{1}{2}}\right).
\end{align*}
Otherwise, if~$\nu=1$, the above bound implies
\begin{align*}
\widehat{\delta}_{1,n}
&<\frac{1}{48}+\widehat{\delta}_{2,n}
<\frac{1}{24}\left(\frac32-n^{-\nicefrac{1}{2}}\right).
\end{align*}
Similarly, we have
\begin{align*}
\widehat{\delta}_{\nu,n}
&>\frac{1}{48}\sum_{k=\nu}^n(k+1)^{-\nicefrac{3}{2}}\\
&>\frac{1}{48}\int_{\nu+1}^{n+2}x^{-\nicefrac{3}{2}}\dx\\
&=\frac{1}{24}\left((\nu+1)^{-\nicefrac{1}{2}}-(n+2)^{-\nicefrac{1}{2}}\right).
\end{align*}
This completes the proof.
\end{proof}

For~$\nu=1$ the above result implies the identity
\begin{equation}\label{eq:nu1}
\Sigma(n)
=A(n)-\frac{1}{6n}-\frac{\delta_{1,n}}{24n},
\end{equation}
which will be crucial in the analysis below. 

\begin{remark}
{\rm
Proceeding in the same way as in the proof of Theorem~\ref{thm:r2}, a formula for the more general case of the arithmetic mean value of the $r$-th roots of the first~$n$ integers, with~$r\ge1$, can be derived: More precisely, for any~$\nu,n\in\mathbb{N}$, with~$\nu< n$, there holds
\[
\sum_{k=\nu}^{n}k^{\nicefrac{1}{r}}
=\frac{r}{r+1}(n+1)^{\nicefrac{1}{r}}\left(n+\frac{1-\nicefrac{1}{r}}{2}\right)
-\frac{r}{r+1}\nu^{\nicefrac{1}{r}}\left(\nu-\frac{1+\nicefrac{1}{r}}{2}\right)
-\frac{\delta_{\nu,n,r}}{12r},
\]
with~$\delta_{\nu,n,1}=0$ (i.e., for~$r=1$), and $\sigma_r(\nu+2,n+2)
<\delta_{\nu,n,r}<\sigma_r(\nu,n)$ for~$r>1$, where
\[
\sigma_r(\nu,n)=
\begin{cases}
2-\nicefrac{1}{r}-n^{-1+\nicefrac{1}{r}}&\text{if }\nu=1,\\
(\nu-1)^{-1+\nicefrac{1}{r}}-n^{-1+\nicefrac{1}{r}}&\text{if }\nu\ge 2.
\end{cases}
\]
} 
\end{remark}

\section{Proof of Theorem~\ref{thm:main}}

The proof of Theorem~\ref{thm:main} is based on the ensuing two auxiliary results. The first lemma provides tight upper and lower bounds on~$A$ from~\eqref{eq:B}. The purpose of the second lemma is to identify any points where the integer part of~$A$ changes.

\begin{lemma}\label{lem:2}
The function~$A$ from~\eqref{eq:B} is strictly monotone increasing for~$x\ge 1$. Furthermore, there hold the bounds
\begin{align}\label{eq:aux29a}
A(x)&<\frac23\sqrt{x+2}&\text{for } x\ge 2,
\intertext{and}
\label{eq:aux29b}
A(x)&>\frac23\sqrt{x+\frac54}+\frac{1}{4x}&\text{for } 
x\ge 6.
\end{align}
\end{lemma}

\begin{proof}
The strict monotonicity of~$A$ follows directly from the fact that
\[
A'(x)=\frac{4x^2-x-2}{12x^2\sqrt{x+1}}>0,
\]
for any~$x\ge 1$. Furthermore, we notice that the graph of~$A$ and of the function
\[
h(x)=\frac23\sqrt{x+2},
\]
which is the upper bound in~\eqref{eq:aux29a}, have exactly one positive intersection point at~$x^\star=\nicefrac{(9+\sqrt{113})}{16}<2$. Moreover, choosing~$x=2>x^\star$, for instance, we have
\[
A(2)=\frac{3\sqrt3}{4}<\frac43=h(2).
\]
Thus, we conclude that~$A(x)<h(x)$ for any~$x>x^\star$; this yields~\eqref{eq:aux29a}. The lower bound~\eqref{eq:aux29b} follows from an analogous argument.
\end{proof}

\begin{lemma}\label{lem:3}
For any~$m\in\mathbb{N}_0$, let $\alpha(m)=\nicefrac94(m+1)^2-2$.
Then, for~$n,m\in\mathbb{N}$, with~$n\le\alpha(m)$, there holds that $A(n)< m+1$. Conversely, we have
$A(n)-\nicefrac{1}{4n}>m+1$, for any integer~$n>\alpha(m)$.
\end{lemma}

\begin{proof}
Consider~$m,n\in\mathbb{N}$ such that~$1\le n\le\alpha(m)$. Applying the monotonicity of~$A$, cf.~Lemma~\ref{lem:2}, together with~\eqref{eq:aux29a}, we infer
\[
A(n)\le A(\alpha(m))<\frac23\sqrt{\alpha(m)+2}=m+1.
\]
Furthermore, if~$m=2s$, with~$s\in\mathbb{N}$, is even, then we have $\alpha(2s)=9s^2+9s+\nicefrac14$; moreover, if~$m=2s-1$, with~$s\in\mathbb{N}$, is odd, then there holds~$\alpha(2s-1)=9s^2-2$. In particular, we conclude that $n\ge \alpha(m)+\nicefrac34$ for any $n\in\mathbb{N}$ with~$n>\alpha(m)$. Then, involving~\eqref{eq:aux29b}, it follows that
\begin{align*}
A(n)-\frac{1}{4n}
>\frac23\sqrt{n+\frac54}
\ge\frac23\sqrt{\alpha(m)+2}
\ge m+1,
\end{align*}
which yields the lemma. 
\end{proof}

\subsection*{Proof of Theorem~\ref{thm:main}}
We are now ready to prove the identity~\eqref{eq:main}. To this end, given~$n\in\mathbb{N}$, we define
\[
m:=\min\{k\in\mathbb{N}:\,\alpha(k)\ge n\}\in\mathbb{N}.
\]
Evidently, there holds $n\le\alpha(m)$ as well as~$n>\alpha(m-1)$.
By virtue of~\eqref{eq:nu1} and due to Lemma~\ref{lem:3}, there holds
\[
1\le\Sigma(n)< A(n)< m+1.
\]
If~$m=1$, the proof of the theorem is complete. Otherwise, if~$m\ge 2$, then by means of Theorem~\ref{thm:r2}, we notice that~$\delta_{1,n}<\nicefrac32$. Then, recalling~\eqref{eq:nu1}, this leads to
\[
\Sigma(n)>A(n)-\frac{1}{6n}-\frac{1}{16n}=A(n)-\frac{11}{48n}.
\]
Hence, upon employing~Lemma~\ref{lem:3} (with~$n>\alpha(m-1)$), it follows that
\[
\Sigma(n)>A(n)-\frac{1}{4n}> m.
\]
Combining the above estimates, we deduce the bounds
\[
m<\Sigma(n)<A(n)<m+1.
\]
This shows~\eqref{eq:main}.

\section{Numerical evaluation of~$\Sigma(n)$}

For large values of~$n$ the straightforward computation of~$\Sigma(n)$, i.e., simply adding the numbers $\sqrt1$, $\sqrt2$, $\sqrt3$,\ldots, $\sqrt{n}$, and dividing by~$n$, is computationally slow and prone to roundoff errors. For this reason, we propose an alternative approach: if~$n\gg 1$, we choose~$\nu\in\mathbb{N}$, with~$\nu+1<n$, of moderate size (so that the numerical evaluation of~$\Sigma(\nu)$ is well-conditioned and accurate). Then, we write
\[
\Sigma(n)=\frac{1}{n}\left(\nu\Sigma(\nu)+\sum_{k=\nu+1}^n\sqrt{k}\right).
\]
Here, employing Theorem~\ref{thm:r2}, we notice that
\[
\sum_{k=\nu+1}^n \sqrt{k}=nA(n)-\nu A(\nu)-\frac{\delta_{\nu+1,n}}{24}.
\]
Hence, upon defining the approximation
\[
\widetilde{\Sigma}(\nu,n):=\frac{1}{n}\left(nA(n)+\nu\Sigma(\nu)-\nu A(\nu)\right),
\]
it follows that
\[
\left|\Sigma(n)-\widetilde\Sigma(\nu,n)\right|\le\frac{\delta_{\nu+1,n}}{24n},
\]
with
\[
\delta_{\nu+1,n}
<\sigma(\nu+1,n)
=\nu^{-\nicefrac{1}{2}}-n^{-\nicefrac{1}{2}}
=n^{-\nicefrac{1}{2}}\left(\left(\frac{\nu}{n}\right)^{-\nicefrac{1}{2}}-1\right);
\]
cf.~\eqref{eq:delta}. In this way, we infer the error estimate
\begin{equation}\label{eq:error}
\left|\Sigma(n)-\widetilde\Sigma(\nu,n)\right|
<\frac{n^{-\nicefrac{3}{2}}}{24}\left(\left(\frac{\nu}{n}\right)^{-\nicefrac12}-1\right).
\end{equation}
In particular, given a prescribed tolerance~$\epsilon>0$, we require
\[
\nu=\intu{n\left(24\epsilon n^{\nicefrac{3}{2}}+1\right)^{-2}},
\]
with~$\nu\le n-2$, in order to deduce the guaranteed bound
\[
\left|\Sigma(n)-\widetilde\Sigma(\nu,n)\right|<\epsilon.
\]
To give an example, we consider~$n=10^7$. Then, choosing~$\nu=100$ leads to the numerical value
\[
\widetilde\Sigma(n)=2108.1852648724285\ldots
\]
In this particular case, the error bound~\eqref{eq:error} gives
\[
\left|\Sigma(n)-\widetilde\Sigma(\nu,n)\right|\le 4.1535\times10^{-10}.
\]
This estimate is fairly sharp; indeed, the true error is approximately~$4.1328\times 10^{-10}$.

\bibliographystyle{amsalpha}
\bibliography{myrefs}

\end{document}